\documentclass[12pt,letterpaper]{amsart}
\usepackage{geometry}
\usepackage{hyperref,amsmath}
\usepackage{mathpazo}
\usepackage{amssymb,url,setspace,xcolor}
\newtheorem{theorem}{Theorem}
\newtheorem{definition}{Definition}
\newtheorem{proposition}{Proposition}
\newtheorem{lemma}{Lemma}
\newtheorem{corollary}{Corollary}

\newtheorem{remark}{Remark}
\onehalfspace

\title[]{Compactness of Composition Operators on the Bergman space of bounded  pseudoconvex domains in $\mathbb{C}^n$}

\author{Timothy G. Clos}

\date{\today}

\begin{document}

 \maketitle

\begin{abstract}
  We study the compactness of composition operators on the Bergman spaces of certain bounded pseudoconvex domains in $\mathbb{C}^n$ with non-trivial analytic disks contained in the boundary.  As a consequence we characterize that compactness of the composition operator with a continuous symbol (up to the closure) on the Bergman space of the polydisk.
\end{abstract}

\section{Introduction}
Let $\Omega\subset \mathbb{C}^n$ be a bounded convex domain.  Let $\mathcal{O}(\Omega)$ be the set of all holomorphic functions from $\Omega$ into $\mathbb{C}$.  Let $V$ be the Lebesgue volume measure on $\Omega$.  For $p\in [1,\infty)$ we define
\[A^p(\Omega):=\{f\in \mathcal{O}(\Omega):\int_{\Omega}|f|^p dV<\infty\}\] to be the $p$-Bergman space.  We denote the norm
as \[\|f\|_{p,\Omega}:=\left(\int_{\Omega}|f|^p dV\right)^{\frac{1}{p}}.\]  This paper will consider the $2$-Bergman space, which for brevity is denoted as the Bergman space.
Let $\phi:\Omega\rightarrow \Omega$ be holomorphic on $\Omega$.  That is, holomorphic in each coordinate function. Then we define the composition operator with symbol $\phi$ as
\[C_{\phi}(f)=f\circ \phi\] for all $f\in A^p(\Omega)$.  For Banach spaces $X$ and $Y$, we say a linear operator $T:X\rightarrow Y$ is compact if $T(\{x\in X:\|x\|<1\})$ is relatively compact in the norm topology on $Y$.  If $X$ is a Hilbert space then we can characterize compactness of linear operator $T:X\rightarrow X$ in terms of weakly convergent sequences.  A set $\nabla\subset \mathbb{C}^n$ is called an analytic disk if there exists holomorphic functions $f_j:\mathbb{D}\rightarrow \mathbb{C}$ for $j\in \{1,2,...,n\}$ and $F(\mathbb{D})=\nabla$ where $F=(f_1,...,f_n)$.  If $\nabla$ is not a single point, then $\nabla$ is said to be a non-trivial or non-degenerate analytic disk.  We define the collection of all non-trivial analytic disks in $b\Omega$ to be \[\mathcal{L}:=\bigcup\{F(\mathbb{D})|\,\,F:\mathbb{D}\rightarrow b\Omega\,,\text{holomorphic and non constant}\}.\] 

As we shall see, analytic structure in the boundary of the domain will play a crucial role in the proof of the main result.  The main theorem does not generalize easily to bounded pseudoconvex domains because of the special geometry of non-degenerate analytic disks in the boundary of the domain.  Namely the following lemma states that on convex domains, non-degenerate analytic disks have a nice form.  

\begin{lemma}[\cite{cucksahut}]\label{affine}
Let $\Omega\subset \mathbb{C}^n$ be a bounded convex domain.  Let $\nabla\subset b\Omega$ be a non-degenerate analytic disk. Then there exists a complex line $V$ so that the convex hull of $\overline{\nabla}$ is contained in $V$. 
\end{lemma}

The following definition is needed for the statement of the main result.

\begin{definition}
A bounded domain $\Omega\subset \mathbb{C}^n$ satisfies the limit point disk condition if $p\in \overline{\mathcal{L}}$ then there exists a non-degenerate analytic disk $\nabla_p\subset b\Omega$ so that $p\in \overline{\nabla_p}$.
\end{definition}

\section{Some Background and Main Results}
Compactness of composition operators was studied on the unit disk in $\mathbb{D}$ in the article \cite{ghatagecompact}. Here, the authors of \cite{ghatagecompact} study the angular derivative of the symbol near the boundary and obtain a compactness result.  They then construct a counterexample to show that the converse of their theorem does not hold true.  The authors of 
\cite{ghatageclosed} studied the closed range property of composition operators on the unit disk.  Work on essential norm estimates and compactness of composition operators was studied on the ball in $\mathbb{C}^n$ and on the unit disk in $\mathbb{C}$ by \cite{CowenMacCluer} and \cite{HMW}.  On more general bounded strongly pseudoconvex domains in $\mathbb{C}^n$, \cite{cuckruhanseveral} studied the essential norm of the composition operator in terms of the behavior of the norm of the normalized Bergman kernel composed with the symbol.   

 In the work \cite{zhu}, Zhu uses the normalized Bergman kernel as a weakly convergent sequence.  Since these can be explicitly computed on the ball, this results in a non-trivial characterization of compactness.  Zhu in \cite{zhu} also extends these results to bounded strongly pseudoconvex domains.  In the absence of such explicit Bergman kernels (on general bounded convex domains), we construct weakly convergent sequences using the convexity of the domain in a significant way, and use estimates on the normalized Bergman kernel near strongly pseudoconvex points.  Here, $K_{z}$ denotes the Bergman kernel of $\Omega$, $V$ is the Lebesgue volume measure, and $B(z,r)$ are Euclidean balls centered at $z$ with radius $r$.  The following two thereoms are the main results.
This first result states that if $C_{\phi}$ is compact, $\Omega$ is convex, and $\phi\in C(\overline{\Omega})$ then $\phi(b\Omega)$ stays away from $\overline{\mathcal{L}}$.

\begin{theorem}\label{analytic disk}
Let $\Omega\subset \mathbb{C}^n$ be a bounded convex domain satisfying the limit point disk condition.  Let $\phi:\Omega\rightarrow \Omega$ be a holomorphic self map, $\phi\in C(\overline{\Omega})$, and suppose $C_{\phi}$ is compact on $A^2(\Omega)$.  Then,
$d(\phi(b\Omega),\overline{\mathcal{L}})>0$.

\end{theorem}

Here, $d(.,\overline{\mathcal{L}})$ is the Euclidean distance to $\overline{\mathcal{L}}$.  This next theorem gives us an estimate on how fast $\|K_{\phi(\zeta)}\|$ goes to infinity as $\phi(\zeta)\rightarrow b\Omega$, assuming $C_{\phi}$ is compact on $A^2(\Omega)$.  Intuitively, it states that the volume $V(B(\zeta,r_{\zeta}))$ goes to zero faster than $\|K_{\phi(\zeta)}\|^2$ goes to infinity.

\begin{theorem}\label{main}
  Let $\Omega\subset \mathbb{C}^n$ be a smooth bounded pseudoconvex domain.  Let $\phi:\Omega\rightarrow \Omega$ be a holomorphic self-map.  If $C_{\phi}$ is compact on $A^2(\Omega)$, then for every $p\in b\Omega$,
\[\lim_{\zeta\in \Omega\,\,,\zeta\rightarrow p}\|K_{\phi(\zeta)}\|_{L^2(\Omega)}(V(B(\zeta,r_{\zeta})))^{\frac{1}{2}}=0.\]  For any $r_{\zeta}>0$, $r_{\zeta}\rightarrow 0$ as $\zeta\rightarrow b\Omega$, and $B(\zeta,r_{\zeta})\subset \Omega$.
\end{theorem}

\begin{remark}
In the case where $\Omega$ is smooth, bounded, and strongly pseudoconvex, we can use Theorem \ref{main} to recover the following theorem seen in \cite{zhu}.

\end{remark}

\begin{theorem}[\cite{zhu}]\label{zhuT}
Let $\Omega\subset \mathbb{C}^n$ be a smooth bounded strongly pseudoconvex domain.   Let $\phi:\Omega\rightarrow \Omega$ be a holomorphic self-map, $\phi(z_1,...z_n)=(\phi_1(z_1,...,z_n),...,\phi_n(z_1,...,z_n))$.  Suppose $C_{\phi}$ is compact on $A^2(\Omega)$, and $p\in b\Omega$.  Then the angular derivative does not exist at $p$.  That is,
\[  \lim_{\zeta\in \Omega\,\,,\zeta\rightarrow p}\frac{d(\zeta,b\Omega)}{d(\phi(\zeta),b\Omega)}=0.\]  Here, $d(.,b\Omega)$ is distance to the boundary.

\end{theorem}

\begin{proof}
Assuming compactness of $C_{\phi}$, we have, by Theorem \ref{main}, 
\[\lim_{\zeta\in \Omega\,\,,\zeta\rightarrow p}\|K_{\phi(\zeta)}\|_{L^2(\Omega)}(V(B(\zeta,r_{\zeta})))^{\frac{1}{2}}=0.\]  Since $\Omega$ is strongly pseudoconvex, $V(\mathcal{B}(\zeta,r))\rightarrow 0$ as $\zeta\rightarrow b\Omega$ where 
$\mathcal{B}(\zeta,r)$ are Bergman metric balls centered at $\zeta$ with radius $r$.  That is, these Bergman metric balls 'shrink' as the center gets closer to the boundary.  Then by strong pseudoconvexity, we have
\[d(\zeta,b\Omega)^{n+1}\approx V(\mathcal{B}(\zeta,r))\leq V(B(\zeta,r_{\zeta}))\] for some $r>0$ sufficiently small.  Furthermore, 
 \[\|K_{\phi(\zeta)}\|^2_{L^2(\Omega)}\approx d(\phi(\zeta),b\Omega)^{-(n+1)}.\]  Thus 
\[   \lim_{\zeta\in \Omega\,\,,\zeta\rightarrow p}\left(\frac{d(\zeta,b\Omega)}{d(\phi(\zeta),b\Omega)}\right)^{\frac{n+1}{2}}  =\lim_{\zeta\in \Omega\,\,,\zeta\rightarrow p}\|K_{\phi(\zeta)}\|_{L^2(\Omega)}(V(B(\zeta,r_{\zeta})))^{\frac{1}{2}}=0.\] 
\end{proof}

The following is a corollary of Theorem \ref{analytic disk} and Proposition \ref{necessary} and concerns the $n$-product of disks.

\begin{corollary}
Let $\mathbb{D}^n$ be the polydisk in $\mathbb{C}^n$.  Suppose $\phi\in C(\overline{\Omega})$ and $\phi:\mathbb{D}^n\rightarrow \mathbb{D}^n$ is a holomorphic self map.  Then $C_{\phi}$ is compact if and only if $\phi(\mathbb{D}^n)$ is relatively compact in $\mathbb{D}^n$.
\end{corollary}

If we assume that the domain is smooth, bounded, convex, and with a certain boundary condition controlling weakly pseudoconvex points, we can get a partial generalization of \cite{zhu}.  That is, we have the following corollary, which is a consequence of Theorem \ref{analytic disk} and Theorem \ref{zhuT}.

\begin{corollary}

Let $\Omega\subset \mathbb{C}^n$ be a smooth bounded convex domain so that 
\[\{(z_1,...,z_n)\in b\Omega: (z_1,...,z_n) \in \text{weakly pseudoconvex}\}=\overline{\mathcal{L}}.\]  Let $\phi(z_1,...,z_n):\Omega\rightarrow \Omega$ be a holomorphic self map and  suppose $\phi\in C(\overline{\Omega})$ (that is, each component function is continuous up to the closure of the domain).  If $C_{\phi}$ is compact on $A^2(\Omega)$ and $p\in b\Omega$ is a strongly pseudoconvex point, then the angular derivative of $\phi$ does not exist at $p$.
That is, 
\[ \lim_{\zeta\in \Omega\,\,,\zeta\rightarrow p}\frac{d(\zeta,b\Omega)}{d(\phi(\zeta),b\Omega)}=0.\]

\end{corollary}

\begin{proof}
Without loss of generality, we may assume $\phi(p)\in b\Omega$, otherwise the conclusion of the corollary is trivial.
By Theorem \ref{analytic disk} and our assumption that \[\{(z_1,...,z_n)\in b\Omega: (z_1,...,z_n) \in \text{weakly pseudoconvex}\}=\overline{\mathcal{L}},\] $\phi(p)$ is in the strongly pseudoconvex part of the boundary.  Thus following the proof of Theorem \ref{zhuT} and using the estimates for the volume of Bergman metric balls near strongly pseudoconvex points, we have our result.

\end{proof}

One crucial result used is the following proposition seen in \cite{ak}.
\begin{proposition}[\cite{ak}]\label{propkernel}
Let $\Omega$ be a smooth bounded pseudoconvex domain.  Then the normalized Bergman kernel 
$k_{\zeta}^{\Omega}\rightarrow 0$ weakly in $A^2(\Omega)$ as $\zeta\rightarrow b\Omega$.
\end{proposition}

\section{Preliminaries}

First we show that compactness of composition operators is invariant under biholomorphisms.

\begin{lemma}\label{lem2}
Let $\Omega_1,\Omega_2\subset \mathbb{C}^n$ for $n\geq 2$ be bounded pseudoconvex domains.  Furthermore, assume there exists a biholomorphism $B:\Omega_1\rightarrow \Omega_2$ so that $B\in C^1(\overline{\Omega_1})$.  Suppose $\phi:=(\phi_1,\phi_2,...,\phi_n):\Omega_2\rightarrow \Omega_2$ is such that the composition operator $C_{\phi}$ is compact on $A^2(\Omega_2)$.  Then, $C_{B^{-1}\circ \phi\circ B}$ is compact on $A^2(\Omega_1)$.  
\end{lemma}

\begin{proof}
Let $g_j\in A^2(\Omega_1)$ so that $g_j\rightarrow 0$ weakly as $j\rightarrow \infty$.  We will use the fact that $g_j\rightarrow 0$ weakly in $A^2(\Omega_1)$ as $j\rightarrow \infty$ if and only if $\|g_j\|$ is a bounded sequence in $j$ and $g_j\rightarrow 0$ uniformly on compact subsets of $\Omega_1$.  This fact appears as \cite[lemma 3.5]{cuckruhanseveral}. Therefore, $\|g_j\|$ is uniformly bounded in $j$ and $g_j\rightarrow 0$ uniformly on compact subsets of $\Omega_1$.
Then define $h_j:=g_j\circ B^{-1}\in  A^2(\Omega_2)$.  Then using a change of coordinates, one can show $\|h_j\|$ is uniformly bounded in $j$ and $h_j\rightarrow 0$ uniformly on compact subsets of $\Omega_2$.  Therefore, by \cite[lemma 3.5]{cuckruhanseveral}, $h_j\rightarrow 0$ weakly as $j\rightarrow \infty$.  
Then we have,
\begin{align*}
&\|C_{B^{-1}\circ \phi\circ B}(g_j)\|_{2,\Omega_1}^2\\
&=\|h_j\circ \phi\circ B\|_{2,\Omega_1}^2\\
&\leq \sup\{|J(B^{-1})(z)|^2:z\in \Omega_2\}\|C_{\phi}(h_j)\|_{2,\Omega_2}^2
\end{align*}

This shows that $C_{B^{-1}\circ\phi\circ B}$ is compact on $A^2(\Omega_1)$.

\end{proof}

\begin{proposition}\label{necessary}

Let $\Omega\subset \mathbb{C}^n$ be a bounded pseudoconvex domain.  Suppose $\phi$ is a holomorphic self-map on $\Omega$ so that $\overline{\phi(\Omega)}\subset \Omega$.  Then $C_{\phi}$ is compact on $A^2(\Omega)$.  

\end{proposition}

\begin{proof}
To prove compactness of $C_{\phi}$, it suffices to show that the image of a weakly convergent sequence in $A^2(\Omega)$ is strongly convergent.  Let $\{g_j\}_{j\in \mathbb{N}}\subset A^2(\Omega)$ converge to $0$ weakly as $j\rightarrow \infty$.  Then by 
\cite[lemma 3.5]{cuckruhanseveral}, $\|g_j\|_{L^2(\Omega)}$ is bounded and $g_j\rightarrow 0$ uniformly on compact subsets of $\Omega$.  We let $\nu:=V\circ \phi^{-1}$ be the pullback measure.  Since $\overline{\phi(\Omega)}$ is compactly contained in $\Omega$, there exists a compact set $\widetilde{\Omega}\subset \Omega$ so that the support of $\nu$ is contained in $\widetilde{\Omega}$.  Thus
\[\int_{\Omega}|g_j\circ \phi|^2dV=\int_{\Omega}|g_j|^2d\nu=\int_{\widetilde{\Omega}}|g_j|^2d\nu.\]  Since $\widetilde{\Omega}$ is compact and $g_j\rightarrow 0$ uniformly on $\widetilde{\Omega}$ as $j\rightarrow \infty$, we have that $C_{\phi}g_j\rightarrow 0$ as $j\rightarrow \infty$ in norm.

\end{proof}

\section{Proof of Main Results}

\begin{proof}[Proof of Theorem \ref{analytic disk}]
One implication for Theorem \ref{main} is entirely contained in Proposition \ref{necessary}.  Therefore, it suffices to assume $C_{\phi}$ is compact on $A^2(\Omega)$ and $\phi^{-1}( b\Omega)\neq \emptyset$.  Let $p\in \phi^{-1}( b\Omega)$.
The first claim is that $\phi(p)$ cannot be contained in a non-trivial analytic disk in $b\Omega$.  For the sake of obtaining a contradiction, suppose $\phi(p)\in \overline{\nabla}\subset b\Omega$ for some non-degenerate analytic disk $\nabla$.   Furthermore, by the convexity of $\Omega$, we may assume $\phi(p)=(0,0,...,0)$, $\nabla\subset \{Re(z_n)=0\}$, and $\Omega\subset \{Re(z_n)>0\}$.  Since $\Omega$ satisfies the limit point disk condition, we may assume there exists a non-degenerate analytic disk $\nabla\subset b\Omega$ and $p\in \overline{\nabla}$.  Then, by Lemma \ref{affine}, $\nabla\subset L$ for some complex line $L$.  Then we construct $g_j$ as in \cite[Proof of Theorem 2]{ClosSahutoglu}.  That is, by rotating and translating the domain, we may assume $q=(0,0,0,...,0)$ and $\Omega\subset \{Re(z_n)>0\}$ and $L\subset \{Re(z_n)>0\}$.  Then for $\delta_j>0$ and $\delta_j\rightarrow 0$ as $j\rightarrow \infty$, we define
\[g_j(z_1,...,z_n):=\frac{1}{\left(z_n-i\delta_j\right)^n}.\]
Then by \cite[Proof of Theorem 2]{ClosSahutoglu}, $g_j\rightarrow 0$ weakly in $A^2(\Omega)$ as $j\rightarrow \infty$ and 
for every open neighbourhood $U$ of $(0,0,...,0)$ there exists $\alpha>0$ so that $\|g_j\|_{L^2(U\cap \Omega)}\geq \alpha$.  One can see this via a continuity argument.   We let $\varepsilon>0$ be sufficiently small and let $dV_{\phi}$ be the pullback measure associated with $\phi$.  Then for \[g_j(z_1,...,z_n):=(z_n-i\delta_j)^{-n}\] we have, 
\begin{align*}
&\int_{\Omega}|g_j\circ \phi|^2dV=\int_{\Omega}|g_j|^2dV_{\phi}\geq\int_{\Omega\cap \mathbb{D}^n_{\varepsilon}}|g_j|^2dV_{\phi}\\
&\geq \int_{\Omega\cap \mathbb{D}^n_{\varepsilon}}(\varepsilon+\delta_j)^{-2n}dV_{\phi}\geq\int_{\Omega\cap \phi^{-1}(\mathbb{D}^n_{\varepsilon})}(\varepsilon+\delta_j)^{-2n}dV\geq M>0
\end{align*}

This contradicts the compactness of $C_{\phi}$.  Therefore, $\phi(p)$ cannot be contained in the closure of a non-degenerate analytic disk.      
\end{proof}

\begin{proof}[Proof of Theorem \ref{main}]
Let $p_j\in \Omega$ and $p_j\rightarrow p\in b\Omega$ as $j\rightarrow \infty$.  Also, choose $r_j>0$ so that $r_j\rightarrow 0$ as $j\rightarrow \infty$ and $B(p_j,r_j)\subset \Omega$ for all $j\in \mathbb{N}$.  Without loss of generality, we may also assume $\phi(p_j)\rightarrow b\Omega$ as $j\rightarrow \infty$.
We have, using the Mean Value theorem for holomorphic functions and the Cauchy-Schwartz inequality,
\begin{align*}
&1=\frac{1}{\|K_{\phi(p_j)}\|^2}K_{\phi(p_j)}(\phi(p_j))\\
&\leq \frac{1}{\|K_{\phi(p_j)}\|^2}\frac{1}{V(B(p_j,r_j))}\int_{D(p_j,r_j)}|K_{\phi(p_j)}(\phi(w))|dV(w)\\
&\leq \frac{1}{\|K_{\phi(p_j)}\|^2}\frac{1}{V(B(p_j,r_j))}\left(\int_{D(p_j,r_j)}|K_{\phi(p_j)}(\phi(w))|^2dV(w)\right)^{\frac{1}{2}}(V(B(p_j,r_j)))^{\frac{1}{2}}
\end{align*}

Thus rearranging we have 
\[\|K_{\phi(p_j)}\|(V(B(p_j,r_j)))^{\frac{1}{2}}\leq \left(\int_{D(p_j,r_j)}|k_{\phi(p_j)}(\phi(w))|^2dV(w)\right)^{\frac{1}{2}}.\]  Then by Proposition \ref{propkernel} and the compactness of $C_{\phi}$, we have our conclusion.    

\end{proof}

It would be interesting say whether these techniques generalize to symbols of less regularity.  For instance, one can try to relate compactness of composition operators to angular derivative (see \cite{zhu}) of symbols on bounded convex domains.

\section{Aknowlegments}
I wish to thank S\"{o}nmez \c{S}ahuto\u{g}lu and Trieu Le for useful conversations and comments on a preliminary version of this manuscript.

\bibliographystyle{amsalpha}
\bibliography{refscomp}

\end{document}